\documentclass[12pt]{amsart}
\usepackage{fixltx2e}
\usepackage{enumerate}
\usepackage[a4paper]{geometry}
\usepackage[latin1]{inputenc}
\usepackage[T1]{fontenc}
\usepackage{mathtools}
\usepackage{amssymb}
\usepackage{mathrsfs}
\usepackage[english]{babel}
\usepackage{hyperref}
\usepackage{microtype}
\hypersetup{%
colorlinks=true,%
linkcolor=black,%
anchorcolor=black,%
citecolor=black,%
menucolor=black,%
urlcolor=black,%
baseurl={http://carva.org/emmanuel.royer},%
pdftitle={Non-vanishing and sign changes of Hecke eigenvalues for Siegel cusp forms of genus two (with an Appendix by E. Kowalski and A. Saha)},%
pdfauthor={Emmanuel Royer (Universit\'e Blaise Pascal, CNRS), Jyoti Sengupta (TIFR), Jie Wu (CNRS, Universit\'e de Lorraine), Emmanuel Kowalski (ETH Zurich), Abhishek Saha (University of Bristol)},%
pdfsubject={spinor zeta function, Siegel form, Fourier coefficients, Hecke operators, eigenvalues, non-vanishing},%
pdfkeywords={spinor zeta function, Siegel form, Fourier coefficients, Hecke operators, Hecke eigenvalues, non-vanishing},%
unicode=true,%
pdflang=en,
}
\newcommand*{\A}{\mathbb{A}}
\newcommand*{\F}{\mathbb{F}}
\newcommand*{\N}{\mathbb{N}}
\newcommand*{\Q}{\mathbb{Q}}
\newcommand*{\Z}{\mathbb{Z}}
\DeclarePairedDelimiter\abs{\lvert}{\rvert}
\newcommand*{\dd}{%
  \mathop{\mathrm{d}\null}\mskip-\thinmuskip\mathord{\null}}
  \newcommand*{\ee}{\mathrm{e}}
\DeclarePairedDelimiter\ent{\lfloor}{\rfloor}
\newcommand*{\eps}{\varepsilon}
\DeclareMathOperator{\frob}{\sigma}
\DeclareMathOperator{\Gal}{Gal}
\renewcommand*{\geq}{\geqslant}
\newcommand*{\GSp}{\mathrm{GSp}}
\newcommand*{\GL}{\mathrm{GL}}
  \newcommand*{\ic}{\mathrm{i}}
\renewcommand*{\leq}{\leqslant}
\newcommand*{\PGSp}{\mathrm{PGSp}}
\newcommand*{\prem}{\mathcal{P}}
\newcommand*{\re}{\mathord{R\mkern-1mu e}}
\DeclareMathOperator{\Res}{Res}
\DeclareMathOperator{\sgn}{sgn}
\newcommand*{\Sp}{\mathrm{Sp}}
\newcommand*{\sset}{\smallsetminus}
\DeclareMathOperator{\Tr}{Tr}
\newtheoremstyle{erthm}
  {}
  {}
  {\itshape}
  {}
  {\fontseries{bx}\selectfont\itshape}
  {--}
  { }
  {}
  \newtheoremstyle{errem}
    {}
    {}
    {}
    {}
    {\itshape}
    {--}
    { }
    {}
\theoremstyle{erthm}
\newtheorem{theorem}{Theorem}
\newtheorem{theoremAN}{Theorem}
\newtheorem{lemma}{Lemma}
\newtheorem{proposition}{Proposition}
\theoremstyle{errem}
\newtheorem{remark}{Remark}
\newtheorem{remarkAN}{Remark}
\title[Non-vanishing and sign changes of Hecke eigenvalues]
{Non-vanishing and sign changes of Hecke eigenvalues for Siegel cusp
  forms of genus two (with an Appendix by E. Kowalski and A. Saha)}
\author[E. Royer]{Emmanuel Royer}
\address{%
Emmanuel Royer\\
Clermont Universit\'e\\
Universit\'e Blaise Pascal\\
Laboratoire de math\'ematiques\\
BP 10448\\
F-63000 Clermont-Ferrand\\
France %
}
\curraddr{%
Emmanuel Royer\\
Universit\'e Blaise Pascal\\
Laboratoire de math\'ematiques\\
Les C\'ezeaux\\
BP 80026\\
F-63171 Aubi\`ere Cedex\\
France %
}
\email{{emmanuel.royer@math.univ-bpclermont.fr}}
\author[J. Sengupta]{Jyoti Sengupta}
\address{Jyoti Sengupta\\
School of Mathematics \\
T.I.F.R. \\
Homi Bhabha Road \\
400\ 005 Mumbai\\
India}
\email{sengupta@math.tifr.res.in}
\author[J. Wu]{Jie Wu}
\address{%
Jie Wu\\
CNRS\\
Institut \'Elie Cartan de Lorraine\\
UMR 7502\\
F-54506 Van\-d\oe uvre-l\`es-Nancy\\
France}
\curraddr{%
Université de Lorraine\\
Institut \'Elie Cartan de Lorraine\\
UMR 7502\\
F-54506 Van\-d\oe uvre-l\`es-Nancy\\
France
}
\email{jie.wu@univ-lorraine.fr}
\thanks{The work was supported by a grant from the Indo-French Center for promotion of Advanced Research (CEFIPRA Project No. 4601-2). We thank warmly the referee for his valuable remarks.}
\begin{document}
\mathtoolsset{showonlyrefs,mathic,centercolon}
\begin{abstract}
In this paper, we show that half of non-zero coefficients of the spinor zeta function of a Siegel cusp form of genus \(2\) 
are positive and half are negative. We also prove results concerning the non-vanishing in short intervals and strong cancellation among the coefficients evaluated at powers of a fixed prime. Our results rest on a Serre's type density result established by Kowalski \& Saha in the appendix.%
\end{abstract}
\maketitle
\section{Introduction}

Let \(S_k\) be the space of Siegel cusp forms of integral weight \(k\) on the group \(\Sp_4(\Z)\subset \GL_4(\Q)\) and let \(F\) be a non zero eigenfunction of all the Hecke operators \(T(n)\) (\(n\in\N\)).  As usual denote by \(\lambda_F(n)\) the \(n\)-th normalized Hecke eigenvalue of \(F\).  Let \(\alpha_{0, p},\alpha_{1,p},\alpha_{2, p}\) be the Satake \(p\)-parameters attached to \(F\) normalized by %
\begin{equation}\label{productalphajp}
\alpha_{0, p}^2 \alpha_{1, p}\alpha_{2, p} = 1.
\end{equation}
Let \(\prem\) be the set of prime numbers. Denote by \(Z_F\) the spinor zeta function of \(F\):
\begin{equation}\label{defZFsEuler}
Z_F(s) \coloneqq\prod_{p\in\prem} Z_{F, p}(p^{-s})^{-1}\eqqcolon \sum_{n\geq 1} a_F(n) n^{-s}\quad(\re s> 1),
\end{equation}
with
\begin{equation}\label{defZFpt}  
Z_{F, p}(t)^{-1}\coloneqq(1-\alpha_{0, p}t)(1-\alpha_{0, p}\alpha_{1, p}t)(1-\alpha_{0, p}\alpha_{2, p}t)(1-\alpha_{0, p}\alpha_{1, p}\alpha_{2, p}t). %
\end{equation}

A Siegel form is in the Maass subspace \(S_k^M\) of \(S_k\) if it is a linear combination of Siegel forms \(F\) that are eigenvectors of all the Hecke operators and for which there exists a primitive modular form \(f\) of weight \(2k-2\) such that %
\[ %
Z_F(s)=\zeta\left(s-\tfrac{1}{2}\right)\zeta\left(s+\tfrac{1}{2}\right)L(f,s). %
\]
Here \(L(f,s)\) is the \(L\)-function of \(f\) (note that we normalize all the \(L\)-functions so that the critical strip is \(0\leq\re s\leq 1\) and the functional equation relates the value at \(s\) to the value at \(1-s\)). This happens only if \(k\) is even. The bijective linear application between \(S_k^M\) and the space of modular forms of weight \(2k-2\) is called the Saito-Kurokawa lifting~\cite{MR633910}. The Ramanujan-Petersson conjecture asserts that
\begin{equation}\label{RPConjecture}
\abs{\alpha_{j, p}}=1 \text{ for \(j=0, 1, 2\) and all primes \(p\).}
\end{equation}
It is not true for Siegel Hecke-eigenforms in \(S_k^M\). But, if \(k\) is odd or, if \(k\) is even and the form is in the orthogonal complement of \(S_k^M\), then it has been established by Weissauer~\cite{Weissauer1993}.  It is well known that
\begin{equation}\label{relationaFlambdaF}
\sum_{n\geq 1} \frac{\lambda_F(n)}{n^s}=\frac{Z_F(s)}{\zeta(2s+1)}. %
\end{equation}
From this, it is easy to see that
\begin{equation}\label{lambdaFnaFn}
\lambda_F(n)= \sum_{d^2m = n} \frac{\mu(d)}{d} a_F(m),
\end{equation}
where \(\mu\) is the M\"obius function. Clearly \eqref{RPConjecture} and \eqref{lambdaFnaFn} imply that
\begin{align}
\abs{a_F(n)}&\leq d_4(n)\label{RP-Conjecture}\\
\abs{\lambda_F(n)}&\leq d_5(n)\label{RamanujanConjecture}
\end{align}
for all integers \(n\geq 1\), where \(d_{\ell}(n)\) is the number of way of writing \(n\) as a product of \(\ell\) positive integers.

We are interested in non vanishing and sign changes of \(\lambda_F(n)\). We denote by \(H_k^*\) the set of Siegel cuspidal Hecke-eigenforms of weight \(k\) and genus \(2\) that, if \(k\) is even, are in the orthogonal complement of \(S_k^M\). The forms we consider in this paper all belong to \(H_k^*\).  According to Breulmann~\cite{Breulmann1999}, a Siegel Hecke-eigenform \(F\) is in the Maass space \(S_k^M\) if and only if \(\lambda_F(n)>0\) for all \(n\); see \cite{Breulmann1999}. On the other hand, Kohnen \cite{Kohnen2007} has proved that if \(F\) is not in the Maass space, then the sequence \(\left(\lambda_F(n)\right)_{n\in\N}\) has infinitely many sign changes. Further Das \cite{Das2013} showed that there is a positive proportion of prime numbers \(p\) such that \(\lambda_F(p)>0\) (respectively \(\lambda_F(p)<0\)).

Define
\begin{equation}\label{defNF*pmx}
\mathscr{N}_{F}^{*}(x)
\coloneqq\sum_{\substack{n\leq x\\ \lambda_F(n)\neq0}} 1,
\qquad
\mathscr{N}_{F}^{+}(x)\coloneqq\sum_{\substack{n\leq x\\ \lambda_F(n)>0}} 1,
\qquad
\mathscr{N}_{F}^{-}(x)\coloneqq\sum_{\substack{n\leq x\\ \lambda_F(n)<0}} 1.
\end{equation}
We are interested in asymptotic behaviour of these functions as \(x\to\infty\).  One of our principal tools is a recent result of Kowalski and Saha (see Theorem 1 in the Appendix below).  Let \(F\) be a Siegel cusp form of genus \(2\) and level \(1\) which is a Hecke eigenform. Then for any \(\delta\in (0, \tfrac{1}{10})\), the inequality
\begin{equation}\label{KowalskiSaha}
\#\{p\leq x : \lambda_F(p)=0\}\ll_{F, \delta} \frac{x}{(\log x)^{1+\delta}}
\end{equation}
holds for all \(x\geq 2\).

The first aim of this paper is to prove the following result by using the method of Matom\"aki \& Radziwill \cite{MR14} based on multiplicative function theory.

\begin{theorem}\label{thm1}
Let \(F\) be a non-zero Siegel-Hecke eigenform in \(S_k\) and suppose that either \(k\) is odd, or \(k\) is even and \(F\) is in the orthogonal complement of \(S_k^M\). %
\begin{enumerate}[(i)] %
\item We have
\begin{equation}\label{NonVanishing}
\mathscr{N}_F^{*}(x)= \rho_F x \bigg\{1 + O_F\bigg(\frac{1}{(\log x)^{\delta}}\bigg)\bigg\}
\end{equation}
for \(x\to\infty\), where \(\delta\) is given by \eqref{KowalskiSaha} and
\begin{equation}\label{defrho}
\rho_F \coloneqq\prod_{p\in\prem}\bigg(1-\frac{1}{p}\bigg)\sum_{\nu\geq 0} \frac{\delta_F(p^{\nu})}{p^{\nu}}>0.
\end{equation}
Here \(\delta_F(n)\) is the characteristic function of the \(n\) such that \(\lambda_F(n)\neq0\).
\item 
We have
\begin{equation}\label{LowerNpmShort}
\mathscr{N}_F^{\pm}(x)
= \tfrac{1}{2} \rho_F x \bigg\{1 + O_F\bigg(\frac{1}{(\log x)^{K/16}}\bigg)\bigg\}
\end{equation}
for \(x\to\infty\), where \(K = 0.32867 \dots = -\cos\phi_0\) and \(\phi_0\) is the unique root in \((0, \pi)\) of the equation \(\sin\phi-\phi\cos\phi=\frac{1}{2}\pi\).
\end{enumerate}
\end{theorem}

\begin{remark}
The same proof shows that the results of Theorem \ref{thm1} also hold if we replace \(\lambda_F(n)\) by \(a_F(n)\). A related short interval result has been obtained by Royer, Sengupta and Wu \cite{RSW2014}.
\end{remark}

\begin{remark}
Theorem~\ref{thm1} establishes that half of non-zero coefficients of the spinor zeta function of a Siegel cusp form of genus \(2\) 
are positive and half are negative. The same result in the framework of modular forms has been established (independently) by Elliott \& Kish \cite{EK14} and by Matom{\"a}ki \& Radziwill \cite{MR14}.
\end{remark}

Pitale and Schmidt \cite{PS2008} proved that if \(F\) is not in the Maass subspace, there exists an infinite set of prime numbers \(p\) so that there are infinitely many \(\nu\) with \(\lambda_F(p^\nu)>0\) and infinitely many \(\nu\) with \(\lambda_F(p^\nu)<0\). Define the parameters \(a\) and \(b\) by %
\begin{equation}\label{defab}
a \coloneqq\alpha_{0,p}, \quad a^{-1} \coloneqq\alpha_{0,p}\alpha_{1,p}\alpha_{2,p},
\quad
b \coloneqq\alpha_{0,p}\alpha_{1,p},
\quad
b^{-1} \coloneqq\alpha_{0,p}\alpha_{2,p}. 
\end{equation}
We also use the notation 
\begin{equation}\label{deftatb}
t_a \coloneqq a+a^{-1}
\qquad\text{and}\qquad
t_b \coloneqq b+b^{-1}.
\end{equation}
The following theorem gives a quantitative description of Pitale and Schmidt's result.

\begin{theorem}\label{thm2}
Let \(F\) be a non-zero Siegel-Hecke eigenform in \(S_k\) and suppose that either \(k\) is odd, or \(k\) is even and \(F\) is in the orthogonal complement of \(S_k^M\). %
\begin{enumerate}[(i)]
\item 
Let \(p\) be a prime number such that \(1, a^2, a^{-2}, b^2, b^{-2}, ab, (ab)^{-1}, ab^{-1}, a^{-1}b\) are different. We have
\begin{equation}\label{SecondMoment}
\sum_{\nu\colon p^{\nu}\leq x} \lambda_F(p^{\nu})^2 \log\bigg(\frac{x}{p^{\nu}}\bigg)= \frac{C_{F,p}}{\log p} (\log x)^2 + O_{F, p}(\log x)
\end{equation}
for all \(x\geq 2\), where 
\begin{equation}\label{defCFp}
C_{F,p}\coloneqq\frac{2}{(t_a-t_b)^2}\bigg\{\bigg(\frac{1}{4-t_a^2}+\frac{1}{4-t_b^2}\bigg)\bigg(1-\frac{1}{p}\bigg)^2 + \frac{2}{p}\bigg\}>0.
\end{equation}
\item 
Let \(p\) be a prime number such that \(1, a, a^{-1}, b, b^{-1}\) are different. We have
\begin{equation}\label{FirstMoment}
\sum_{\nu\colon p^{\nu}\leq x} \lambda_F(p^{\nu}) \log\bigg(\frac{x}{p^{\nu}}\bigg)\ll_{F, p} \log x
\end{equation}
for all \(x\geq 2\).
\end{enumerate}
\end{theorem}

\begin{remark}
\begin{enumerate}[(i)]
\item Theorem \ref{thm2} shows that there is a very strong cancellation among the \(\lambda_F(p^\nu)\) and there are few \(\nu\) such that \(\lambda_F(p^{\nu})\) is large.
\item Our approach allows us to establish a similar result for \(a_F(p^{\nu})\) (even more easily).
\end{enumerate}
\end{remark}

In order to measure the non vanishing of \(\lambda_F(n)\), we introduce, as in \cite{Serre1981}, 
\[
i_F(n)\coloneqq\max\{j\geq 1 \colon\lambda_F(n+i) = 0 \;\, \text{for \(0<i\leq j\)}\}
\]
with the convention that \(\max \emptyset = 0\). We hope to get non-trivial bound of type 
\[
i_F(n)\ll_F n^{\theta}
\] 
for some \(\theta<1\) and all \(n\geq 1\). Clearly a stronger form of the problem is to find \(y\) as small as possible (as a function of \(x\), say \(y=x^{\theta}\) with \(\theta<1\)) such that
\[
\#\{x<n\leq x+y\colon\lambda_F(n)\neq0\}\gg y,
\]
where the implied constant can depend on \(F\). 

We can prove the following result by using \(\mathscr{B}\)-free number theory as in \cite{KowalskiRobertWu2007, WuZhai2013}.

\begin{theorem}\label{thm3}
Let \(F\) be a Siegel cusp form of genus \(2\) and level \(1\) which is a Hecke eigenform.
\begin{enumerate}[(i)]
\item
For every \(\eps>0\), \(x\geq x_0(F, \eps)\) and \(y\geq x^{7/17+\eps}\), we have
\[
\#\{x<n\leq x+y\colon \text{\(\mu(n)^2=1\) and \(\lambda_F(n)\neq0\)}\}\gg_{F, \eps} y. 
\]
In particular for any \(\eps>0\) and all \(n\geq 1\), we have
\[
i_F(n)\ll_{F, \eps} n^{7/17+\eps}.
\]
\item
For every \(\eps>0\), \(x\geq x_0(F, \eps)\), \(y\geq x^{17/38+100\eps}\) and \(1\leq a\leq q\leq x^\eps\) with \((a, q)=1\), we have
\[
\#\{x<n\leq x+y\colon\text{\(\mu(n)^2=1\), \(n\equiv a \!\!\!\!\! \pmod{q}\) and \(\lambda_F(n)\neq0\)}\}
\gg_{F, \eps} y/q. 
\]
\end{enumerate}
\end{theorem}

\begin{remark}
According to \eqref{lambdaFnaFn}, we have \(\lambda_F(n)=a_F(n)\) for all square free integers \(n\). The results of Theorem \ref{thm3} also hold if we replace \(\lambda_F(n)\) by \(a_F(n)\).
\end{remark}

\begin{remark}
Theorem \ref{thm3}(i) improves considerably a recent result of Das, Kohnen and Sengupta  \cite[Corollary 1.5]{MR3182539}, which requires $\tfrac{31}{32}$ in place of $\tfrac{7}{17}$.
\end{remark}


\section{Proof of Theorem \ref{thm1}}

Firstly we state two results on mean values of multiplicative functions. The first one is a particular case of \cite[Theorem]{HallTenenbaum1991}. In the following, \(p\) stands always for a prime number. %

\begin{lemma}\label{HT1991}
Let \(g\) be a real multiplicative function such that \(\abs{g(n)}\leq 1\). Then we have
\[
\sum_{n\leq x} g(n)\ll x \exp\bigg\{-K\sum_{p\leq x} \frac{1-g(p)}{p}\bigg\}
\]
where \(K = 0.32867 \dots = -\cos\phi_0\) and \(\phi_0\) is the unique root in \((0, \pi)\) of the equation \(\sin\phi-\phi\cos\phi=\frac{1}{2}\pi\).
\end{lemma}

The second lemma is an unpublished result of Halberstam. A complete proof has been given by Song \cite[Theorem A]{Song2001}. A more general result has been proved by Liu and Wu \cite[Theorem 2]{LiuWu2014}.

\begin{lemma}\label{LW2014}
Let \(g\) be a non-negative multiplicative function satisfying the following conditions 
\begin{align}
\sum_{p\leq z} g(p)\log p &= \kappa z + O\bigg(\frac{z}{(\log z)^{\delta}}\bigg) \quad (z\geq 2),
\label{TWcondition1}
\\
\sum_{p, \, \nu\geq 2} \frac{g(p^\nu)}{p^\nu} \log p^\nu&\leq A,
\label{TWcondition2}
\end{align}
where \(A>0\), \(\kappa>0\) and \(\delta>0\) are constants.
Then we have
\begin{equation}\label{AsymptoticSfx}
\sum_{n\leq x} g(n) = C_gx(\log x)^{\kappa-1} \bigg\{1
+ O_{g, \delta}\bigg(\frac{\log_2x}{\log x} + \frac{1}{(\log x)^{\delta}}\bigg)\bigg\},
\end{equation}
where
\[
C_g \coloneqq\prod_p \bigg(1-\frac{1}{p}\bigg)^\kappa 
\sum_{\nu\geq 0} \frac{g(p^{\nu})}{p^{\nu}}. %
\]
\end{lemma}

\begin{lemma}\label{LBWY}
Let \(F\) be a non-zero Siegel-Hecke eigenform in \(S_k\) and suppose that either \(k\) is odd, or \(k\) is even and \(F\) is in the orthogonal complement of \(S_k^M\). Then there are two positive constants \(c_F>0\) and \(x_0(F)>0\) depending on \(F\) only such that
\begin{align}\label{LB1/32}
\sum_{\substack{p\leq x\\ \lambda_F(p)>0}} \log p&\geq \frac{1}{32} x + O_F\big(x\ee^{-c_F\sqrt{\log x}}\big)\\
\shortintertext{and} %
\sum_{\substack{p\leq x\\ \lambda_F(p)<0}} \log p&\geq \frac{1}{32} x + O_F\big(x\ee^{-c_F\sqrt{\log x}}\big)\\
\end{align}
for all \(x\geq x_0(F)\).
\end{lemma}

\begin{proof}
According to \cite[Theorem 5.1.2]{PSS2012}, the transfer of \(F\) is an irreductible unitary cuspidal and self-contragredient automorphic representation of \(\GL_4(\Q)\). Thus we can apply \cite[Theorem 3]{WuYe2007} to write
\begin{align}
\sum_{p\leq x} \lambda_F(p)^2 \log p & = x + O_F\big(x\ee^{-c_F\sqrt{\log x}}\big),
\label{RankinSelbergPNTSiegelForm}
\\
\sum_{p\leq x} \lambda_F(p) \log p& \ll_F x \ee^{-c_F\sqrt{\log x}}
\label{PNTSiegelForm}
\end{align}
for all \(x\geq 2\). In view of \eqref{RP-Conjecture} and the fact that \(d_4(p)=4\), it is clear that \(\abs{\lambda_F(p)}\geq \lambda_F(p)^2/4\). Thus \eqref{RankinSelbergPNTSiegelForm} implies immediately
\begin{equation}\label{LBpi1}
\sum_{p\leq x}\abs{\lambda_F(p)}\log p\geq \tfrac{1}{4}x + O_F\big(x \ee^{-c_F\sqrt{\log x}}\big)\qquad (x\geq x_0(F)).
\end{equation}
Defining
\[
\lambda_F^{\pm}(p) = \frac{\abs{\lambda_F(p)}\pm \lambda_F(p)}{2},
\]
the relations \eqref{LBpi1} and \eqref{PNTSiegelForm} imply that
\[
\sum_{p\leq x} \lambda_F^{\pm}(p) \log p
\geq \tfrac{1}{8} x + O_F\big(x \ee^{-c_F\sqrt{\log x}}\big)
\qquad
(x\geq x_0(F)).
\]
This implies \eqref{LB1/32} since \(0\leq \lambda_F^{\pm}(p)\leq 4\) and %
\[
\lambda_F^{+}(p)\neq0\Leftrightarrow\lambda_F(p)>0\;\text{ and }\;\lambda_F^{-}(p)\neq0\Leftrightarrow\lambda_F(p)<0. %
\]
\end{proof}

Now we are ready to prove Theorem \ref{thm1}. Define
\[
g_F(n) \coloneqq
\begin{cases}
\sgn(\lambda_F(n)) & \text{if \(\lambda_F(n)\neq0\)},
\\
0 & \text{otherwise}.
\end{cases}
\]
By the prime number theorem and \eqref{KowalskiSaha}, we have
\begin{align*}
\sum_{p\leq x}\abs{g_F(p)}\log p 
& = \sum_{p\leq x} \log p - \sum_{\substack{p\leq x\\ \lambda_F(p)=0}} \log p  
\\
& = x +O\bigg(\frac{x}{(\log x)^{\delta}}\bigg).
\end{align*}
This shows that the function \(\abs{g_F}\) satisfies condition \eqref{TWcondition1} of Lemma \ref{LW2014} with \(\kappa=1\). Condition \eqref{TWcondition2} is satisfied trivially. We obtain
\begin{equation}\label{LW}
\sum_{n\leq x}\abs{g_F(n)}= \rho_F x \bigg\{1+O\bigg(\frac{1}{(\log x)^{\delta}}\bigg)\bigg\}.
\end{equation}
This proves \eqref{NonVanishing}, since \(\abs{g_F(n)}=\delta_F(n)\).

On the other hand,  \eqref{LB1/32} of Lemma \ref{LBWY} allows us to deduce
\begin{equation}\label{Das}
\begin{aligned}
\sum_{\substack{p\leq x\\ g_F(p)=-1}} \frac{1}{p}
&= \int_{2-}^x \frac{1}{t\log t} \dd \Big(\sum_{\substack{p\leq t\\ \lambda_F(p)<0}} \log p\Big)
\\
& \geq \frac{1}{32}\log_2x + O_F(1)
\quad
(x\to\infty).
\end{aligned}
\end{equation}
From Lemma \ref{HT1991} and \eqref{Das}, we can deduce that
\begin{equation}\label{HT}
\begin{aligned}
\sum_{n\leq x} g_F(n)
& \ll x\exp\bigg\{-K\sum_{p\leq x} \frac{1-g_F(p)}{p}\bigg\}
\\
& \ll x\exp\bigg\{-2K\sum_{\substack{p\leq x\\ g_F(p)=-1}} \frac{1}{p}\bigg\}
\\
& \ll \frac{x}{(\log x)^{K/16}}. %
\end{aligned}
\end{equation}
Clearly \eqref{LW} and \eqref{HT} imply the required result, since \(\abs{g_F(n)}=\delta_F(n)\) and
\[
\frac{\abs{g_F(n)}+ g_F(n)}{2}
= \begin{cases}
1 & \text{if \(\lambda_F(n)\geq 0\)},
\\
0 & \text{otherwise}.
\end{cases}
\]
and
\[
\frac{\abs{g_F(n)}-g_F(n)}{2}
= \begin{cases}
1 & \text{if \(\lambda_F(n)\leq 0\)},
\\
0 & \text{otherwise}.
\end{cases}
\]
This completes the proof.

\vskip 8mm

\section{Dirichlet series associated to \(\lambda_F(p^\nu)^2\) and \(a_F(p^\nu)^2\)}\label{DS}

As before, let \(F\) be a non zero eigenfunction of all the Hecke operators \(T(n)\), (\(n\in \N\)) and let \(a, a^{-1}, b, b^{-1}\) be defined as in \eqref{defab}. We introduce
\begin{equation}\label{defDF}
\mathcal{D}_F \coloneqq\big\{1, a^2, a^{-2}, b^2, b^{-2}, ab, (ab)^{-1}, ab^{-1}, a^{-1}b\big\}.
\end{equation}
In view of \eqref{RamanujanConjecture}, the Dirichlet series
\begin{equation}\label{defFps}
\mathcal{F}_{p}(s) \coloneqq\sum_{\nu=0}^{\infty} \frac{\lambda_F(p^\nu)^2}{p^{\nu s}}
\end{equation}
is absolutely convergent for \(\re s>0\) and any prime number \(p\). The aim of this section is to give an explicit expression for this Dirichlet series.

\begin{proposition}\label{pro1}
Under the previous notation, we have
\[
\mathcal{F}_{p}(s) 
= (1+p^{-s}) \prod_{\eta\in\mathcal{D}_F}(1-\eta p^{-s})^{-1} \sum_{0\leq i\leq 6} q_i p^{-is}
\]
for \(\re s>0\), where \(t_a, t_b\) are defined as in \eqref{deftatb} and
\begin{align*}
q_0
& \coloneqq 1,
\\
q_1
& \coloneqq t_at_b+2,
\\
q_2 
&\coloneqq 2-(t_a+t_b)^2-2(t_a^2+t_b^2+t_at_b-2)p^{-1}+p^{-2},
\\
q_3
& \coloneqq t_at_b+2+2\left[(t_a+t_b)^2+(t_a^2-2)(t_b^2-2)\right]p^{-1}+(t_at_b+2)p^{-2},
\\
q_4
& \coloneqq 1-2(t_a^2+t_b^2+t_at_b-2)p^{-1}-\left[(t_a+t_b)^2-2\right]p^{-2},
\\
q_5
& \coloneqq(t_at_b+2)p^{-2},
\\
q_6
& \coloneqq p^{-2}. 
\end{align*}
\end{proposition}

\begin{remark}\label{rem_residue}
Let \(p\) be a prime number such that \(1, a^2, a^{-2}, b^2, b^{-2}, ab, (ab)^{-1}, ab^{-1}, a^{-1}b\) are different. The residue at \(0\) of \(\mathcal{F}_{p}\) is %
\[%
\Res\left(\mathcal{F}_p(s),0\right)=\frac{C_{F,p}}{\log p} %
\] 
where %
\[%
C_{F,p}=2\prod_{\substack{\eta\in\mathcal{D}_F\\\eta\neq 1}}(1-\eta)^{-1}\sum_{0\leq i\leq 6}q_i.
\]
This can be simplified in %
\[%
C_{F,p}=\frac{2}{(t_a-t_b)^2}\bigg\{\bigg(\frac{1}{4-t_a^2}+\frac{1}{4-t_b^2}\bigg)\bigg(1-\frac{1}{p}\bigg)^2 + \frac{2}{p}\bigg\}. %
\]
This last expression is clearly positive. In particular \(\mathcal{F}_p\) has a simple pole at \(0\).
\end{remark}

\begin{remark}
The expressions of \(q_0,\dotsc,q_6\) are symmetric in \((t_a,t_b)\). We can then make the change of variable %
\begin{align*}%
u&=t_a+t_b=a+\frac{1}{a}+b+\frac{1}{b}\\
v&=t_at_b+2=\left(a+\frac{1}{a}\right)\left(b+\frac{1}{b}\right)+2.
\end{align*}
We obtain %
\begin{align*}
q_0& = 1,\\
q_1& = v,\\
q_2& = 2-u^2-2(u^2-v)p^{-1}+p^{-2},
\\
q_3& = v-2(u+v)(u-v)p^{-1}+vp^{-2},
\\
q_4& = 1-2(u^2-v)p^{-1}+(2-u^2)p^{-2},
\\
q_5& = vp^{-2},
\\
q_6& = p^{-2}. 
\end{align*}
\end{remark}

Proposition \ref{pro1} is an immediate consequence of \eqref{lambdaFpnu2} and Lemmas \ref{lem3.2}-\ref{lem3.4} below.

According to \cite[Proposition 3.35]{MR1349824}, the generating series of the sequence \(\left(\lambda_F(p^\nu)\right)_{\nu\geq0}\) is %
\[
\sum_{\nu\geq 0}\lambda_F(p^\nu)X^\nu=\frac{1-p^{-1}X^{2}}{(1-aX)(1-a^{-1}X)(1-bX)(1-b^{-1}X)}.
\]
From this identity, Pitale \& Schmidt proved \cite[Proposition 4.1]{PS2009} that
\begin{equation}\label{deflambdapnu}
\lambda_F(p^\nu)
= A(\nu) + (1-p^{-1})\sum_{1\leq j\leq\ent{\nu/2}} A(\nu-2j),
\end{equation}
where 
\begin{equation}\label{defAnu}
A(\nu) \coloneqq\sum_{0\leq i\leq \nu} a^{\nu-i} b^i \sum_{0\leq j\leq \nu} (ab)^{-j}.
\end{equation}

From \eqref{deflambdapnu}, we have
\begin{equation}\label{lambdaFpnu2}
\lambda_F(p^\nu)^2 = A(\nu)^2 + 2(1-p^{-1})B(\nu) + (1-p^{-1})^2C(\nu),
\end{equation}
where
\[
B(\nu) \coloneqq A(\nu) \sum_{j=1}^{\ent{\nu/2}} A(\nu-2j),
\qquad
C(\nu) \coloneqq\Big(\sum_{j=1}^{\ent{\nu/2}} A(\nu-2j)\Big)^2.
\] 

\begin{lemma}\label{lem3.1}
Let \(\nu\geq 1\) and let \(a\) and \(b\) be complex numbers of norm \(1\) such that 
\[
(a-b)(a^2-1)(b^2-1)(ab-1)\neq0.
\]
We have 
\begin{align}
A(\nu)
& = D(a, b) \big(a^{\nu+1}+a^{-\nu-1}-b^{\nu+1}-b^{-\nu-1}\big),
\label{lem3.1.A}
\\
\sum_{j=1}^{\ent{\nu/2}} A(\nu-2j)
& = D(a, b) \bigg(\frac{a^{\nu}-a^{-\nu}}{a-a^{-1}} -\frac{b^{\nu}-b^{-\nu}}{b-b^{-1}}\bigg),
\label{lem3.1.B}
\end{align}
where 
\begin{equation}\label{defCab}
D(a, b) \coloneqq\frac{ab}{(a-b)(ab-1)}.
\end{equation}
\end{lemma}

\begin{proof}
By summing the geometric series in~\eqref{defAnu}, we find that
\[
A(\nu) = \frac{a^{\nu+1}-b^{\nu+1}}{a-b} \cdot \frac{1-(ab)^{-\nu-1}}{1-(ab)^{-1}}.
\]
After simplification, we get \eqref{lem3.1.A}.

From this it is easy to see that, with the notation \(w\coloneqq\ent{\nu/2}\),
\begin{align*}
\frac{1}{D(a, b)}\sum_{j=1}^{w} A(\nu-2j)
& = \sum_{j=1}^{w} \big(a^{\nu-2j+1}+a^{-\nu+2j-1}-b^{\nu-2j+1}-b^{-\nu+2j-1}\big)
\\
& = \frac{a^{\nu}-a^{-\nu}}{a-a^{-1}} -\frac{b^{\nu}-b^{-\nu}}{b-b^{-1}}
- \frac{a^{\nu-2w}-a^{-\nu+2w}}{a-a^{-1}} + \frac{b^{\nu-2w}-b^{-\nu+2w}}{b-b^{-1}}.
\end{align*}
This implies \eqref{lem3.1.B} since
\[
\frac{a^{\nu-2w}-a^{-\nu+2w}}{a-a^{-1}} - \frac{b^{\nu-2w}-b^{-\nu+2w}}{b-b^{-1}}\equiv 0.
\]
\end{proof}

\begin{lemma}\label{lem3.2}
Let \(a\) and \(b\) be complex numbers of norm \(1\) such that 
\[
(a-b)(a^2-1)(b^2-1)(ab-1)\neq0.
\]
For \(\abs{t}<1\), we have 
\begin{align*}
\sum_{\nu=0}^{\infty} \frac{A(\nu)^2}{D(a, b)^2} t^{\nu}
& = \frac{4}{1-t} 
+ \frac{a^2}{1-a^2t} + \frac{a^{-2}}{1-a^{-2}t} 
+ \frac{b^2}{1-b^2t} + \frac{b^{-2}}{1-b^{-2}t}
\\
& \quad
- \frac{2ab}{1-abt} 
- \frac{2(ab)^{-1}}{1-(ab)^{-1}t}
- \frac{2ab^{-1}}{1-ab^{-1}t} 
- \frac{2a^{-1}b}{1-a^{-1}bt}.
\end{align*}
\end{lemma}

\begin{proof}
With the help of \eqref{lem3.1.A} of Lemma \ref{lem3.1}, we have
\begin{align*}
D(a, b)^{-2}A(\nu)^2
& = \big(a^{\nu+1}+a^{-\nu-1}-b^{\nu+1}-b^{-\nu-1}\big)^2
\\
& = 4 + a^{2(\nu+1)} + a^{-2(\nu+1)} + b^{2(\nu+1)} + b^{-2(\nu+1)}
\\
& \quad
- 2(ab)^{\nu+1} - 2(ab)^{-\nu-1} - 2(ab^{-1})^{\nu+1} - 2(a^{-1}b)^{\nu+1}.
\end{align*}
This implies the required formula.
\end{proof}

\begin{lemma}\label{lem3.3}
Let \(a\) and \(b\) be complex numbers of norm \(1\) such that 
\[
(a-b)(a^2-1)(b^2-1)(ab-1)\neq0.
\]
For \(\abs{t}<1\), we have 
\begin{align*}
\sum_{\nu=0}^{\infty} \frac{B(\nu)}{D(a, b)^2}  t^{\nu}
& = - \frac{2}{1-t}
\\
& \quad
+ \frac{a^2}{a^2-1} \cdot \frac{1}{1-a^2t}
- \frac{1}{a^2-1} \cdot \frac{1}{1-a^{-2}t}
\\
& \quad
+ \frac{b^2}{b^2-1} \cdot \frac{1}{1-b^2t}
- \frac{1}{b^2-1} \cdot \frac{1}{1-b^{-2}t}
\\
& \quad
- \frac{a^3b+ab^3-2ab}{(a^2-1)(b^2-1)} \cdot \frac{1}{1-abt}
- \frac{ab^{-1}+a^{-1}b-2ab}{(a^2-1)(b^2-1)} \cdot \frac{1}{1-(ab)^{-1}t}
\\
& \quad
+ \frac{a^3b+ab^{-1}-2ab}{(a^2-1)(b^2-1)} \cdot \frac{1}{1-ab^{-1}t}
+ \frac{ab^3+a^{-1}b-2ab}{(a^2-1)(b^2-1)} \cdot \frac{1}{1-a^{-1}bt}.
\end{align*}
\end{lemma}

\begin{proof}
With the help of \eqref{lem3.1.A} and \eqref{lem3.1.B} of Lemma \ref{lem3.1}, we have
\begin{align*}
\frac{B(\nu)}{D(a, b)^2}
& = (a^{\nu+1}+a^{-\nu-1}-b^{\nu+1}-b^{-\nu-1})\bigg(\frac{a^{\nu}-a^{-\nu}}{a-a^{-1}} -\frac{b^{\nu}-b^{-\nu}}{b-b^{-1}}\bigg).
\\
& = \frac{a^{2\nu+1}-a^{-2\nu-1}-a^{\nu}b^{\nu+1}+a^{-\nu}b^{-\nu-1}-a^{\nu}b^{-\nu-1}+a^{-\nu}b^{\nu+1}}{a-a^{-1}} 
\\
& \quad
+ \frac{b^{2\nu+1}-b^{-2\nu-1} - a^{\nu+1}b^{\nu}+a^{-\nu-1}b^{-\nu} - a^{-\nu-1}b^{\nu}+a^{\nu+1}b^{-\nu}}{b-b^{-1}}
-2
\end{align*}
From this we deduce for \(\abs{t}<1\), 
\begin{align*}
& \sum_{\nu=0}^{\infty} \frac{B(\nu)}{D(a, b)^2} t^{\nu}
\\
& = \frac{a}{a^2-1} 
\bigg(\frac{a}{1-a^2t} - \frac{a^{-1}}{1-a^{-2}t} 
- \frac{b}{1-abt} + \frac{b^{-1}}{1-(ab)^{-1}t}
- \frac{b^{-1}}{1-ab^{-1}t}
+ \frac{b}{1-a^{-1}bt}\bigg)
\\
& + \frac{b}{b^2-1} 
\bigg(
\frac{b}{1-b^2t} - \frac{b^{-1}}{1-b^{-2}t}
- \frac{a}{1-abt} + \frac{a^{-1}}{1-(ab)^{-1}t} 
- \frac{a^{-1}}{1-a^{-1}bt} + \frac{a}{1-ab^{-1}t} 
\bigg)
\\
& - \frac{2}{1-t}\cdot
\end{align*}
This implies the required formula.
\end{proof}

\begin{lemma}\label{lem3.4}
Let \(a\) and \(b\) be complex numbers of norm \(1\) such that 
\[
(a-b)(a^2-1)(b^2-1)(ab-1)\neq0.
\]
For \(\abs{t}<1\), we have 
\begin{align*}
\sum_{\nu=0}^{\infty} \frac{C(\nu)}{D(a, b)^2}
&  t^{\nu}
= - \bigg(\frac{2a^2}{(a^2-1)^2} + \frac{2b^2}{(b^2-1)^2}\bigg) \frac{1}{1-t}
\\
& 
+  \frac{a^2}{(a^2-1)^2}\bigg(\frac{1}{1-a^2t} + \frac{1}{1-a^{-2}t}\bigg)
+ \frac{b^2}{(b^2-1)^2}\bigg(\frac{1}{1-b^2t} + \frac{1}{1-b^{-2}t}\bigg)
\\
& 
- \frac{2ab}{(a^2-1)(b^2-1)}
\bigg(\frac{1}{1-abt} + \frac{1}{1-(ab)^{-1}t} - \frac{1}{1-ab^{-1}t} - \frac{1}{1-a^{-1}bt}\bigg).
\end{align*}
\end{lemma}

\begin{proof}
With the help of \eqref{lem3.1.B} of Lemma \ref{lem3.1}, we have
\begin{align*}
\frac{C(\nu)}{D(a, b)^2}
& = \bigg(\frac{a^{\nu}-a^{-\nu}}{a-a^{-1}}\bigg)^2
+ \bigg(\frac{b^{\nu}-b^{-\nu}}{b-b^{-1}}\bigg)^2
- 2\frac{(a^{\nu}-a^{-\nu})(b^{\nu}-b^{-\nu})}{(a-a^{-1})(b-b^{-1})}
\\
& = \frac{a^{2\nu}-2+a^{-2\nu}}{(a-a^{-1})^2}
+ \frac{b^{2\nu}-2+b^{-2\nu}}{(b-b^{-1})^2}
- 2\frac{(ab)^{\nu}+(ab)^{-\nu}-(ab^{-1})^{\nu}-(a^{-1}b)^{\nu}}{(a-a^{-1})(b-b^{-1})}. 
\end{align*}
From this we deduce for \(\abs{t}<1\),  
\begin{align*}
& \sum_{\nu=0}^{\infty} \frac{C(\nu)}{D(a, b)^2} t^{\nu}
\\
& = \frac{a^2}{(a^2-1)^2}\bigg(\frac{1}{1-a^2t} + \frac{1}{1-a^{-2}t} - \frac{2}{1-t}\bigg)
\\
& \quad
+ \frac{b^2}{(b^2-1)^2}\bigg(\frac{1}{1-b^2t} + \frac{1}{1-b^{-2}t} - \frac{2}{1-t}\bigg)
\\
& \quad
- \frac{2ab}{(a^2-1)(b^2-1)}
\bigg(\frac{1}{1-abt} + \frac{1}{1-(ab)^{-1}t} - \frac{1}{1-ab^{-1}t} - \frac{1}{1-a^{-1}bt}\bigg).
\end{align*}
This implies the required formula.
\end{proof}

An expanded expression for %
\[%
\sum_{\nu\geq 0}\lambda_F(p^\nu)^2t^\nu %
\]
is deduced from~\eqref{lambdaFpnu2} and Lemmas \ref{lem3.2}-\ref{lem3.4}. Factoring this expression, we obtain %
\[%
\sum_{\nu\geq 0}\lambda_F(p^\nu)^2t^\nu=\frac{1+t}{\prod_{\eta\in\mathcal{D}_F}(1-\eta t)}\sum_{i=0}^6q_it^i %
\]
where the \(t_i\) are given in terms of \(a\) and \(b\). It is easy to check that these values are expressed in terms of \(t_a\) and \(t_b\) as in Proposition~\ref{pro1}.

Similar to Proposition \ref{pro1}, we can prove the following result.

\begin{proposition}\label{pro2}
For \(\re s>0\), we have
\[
\sum_{\nu\geq 0} \frac{a_F(p^{\nu})^2}{p^{\nu s}}
= \prod_{0\leq j\leq 8} (1-\eta_jp^{-s})^{-1} \sum_{0\leq j\leq 4} r_j p^{-js}
\]
where
\[
r_0 = r_4=1,
\qquad
r_1 =r_3= t_at_b+2,
\qquad
r_2 = -(t_a+t_b)^2+2.
\]
\end{proposition}

\section{Proof of Theorem \ref{thm2}}

Let \(\mathcal{D}_F\) and \(\eta_j\) be defined as at the beginning of Section \ref{DS}. According to \eqref{RPConjecture}, we have \(\abs{\eta_j}=1\). Thus we can write \(\eta_j = \ee^{\ic\theta_j}\) with \(\theta_j\in (-\pi, \pi]\). We choose a \(\theta_{*}\in (-\pi, \pi]\) such that \(\sin(\pm\theta_j-\theta_{*})\neq0\) for \(0\leq j\leq 8\). Let \(\eps\) be an arbitrarily small positive number and take
\[
T = \frac{2\pi\ent{\log x}+\theta_{*}}{\log p}. %
\]

In view of Proposition \ref{pro1}, we can apply the Perron formula (\cite[Theorem II.2.5]{Te95}) to write
\begin{equation}\label{Perron}
\sum_{p^\nu\leq x} \lambda_F(p^\nu)^2 \log\bigg(\frac{x}{p^{\nu}}\bigg)
= \frac{1}{2\pi\ic} \int_{(\log x)^{-1}-\ic\infty}^{(\log x)^{-1}+\ic\infty} \mathcal{F}_{p}(s) \frac{x^s}{s^2} \dd s.
\end{equation}

We truncate the integral at \(T\). Since \(\abs{1-\eta_j p^{-s}}\geq 1-p^{-(\log x)^{-1}}\gg_p (\log x)^{-1}\) for \(s=(\log x)^{-1}+\ic\tau\) with \(\abs{\tau}\geq T\) and \(0\leq j\leq 8\), we have \(\mathcal{F}_p(s)\ll_{F,p} \log x\) for these values of \(s\). Thus 
\begin{equation}\label{infty}
\frac{1}{2\pi\ic} \int_{\substack{s=(\log x)^{-1}+\ic\tau\\ \abs{\tau}\geq T}} \mathcal{F}_{p}(s) \frac{x^s}{s^2} \dd s\ll_{F, p} \frac{\log x}{T}.
\end{equation}

Now we shift the segment of integration \([(\log x)^{-1}-\ic T, (\log x)^{-1}+\ic T]\) to \([-\eps-\ic T, -\eps+\ic T]\).
The poles of \(\mathcal{F}(s) \frac{x^s}{s^2}\) in the rectangle \(-\eps\leq \re s\leq (\log x)^{-1}\) and \(\abs{\tau}\leq T\) are as follows :
\[
s_{j, \ell} \coloneqq\frac{\theta_j+2\pi\ell}{\log p}\ic
\quad
(0\leq j\leq 8)
\]
where \(\ell\in {\Z}\) such that \(\abs{s_{j, \ell}}<T\). Clearly the number of such poles is bounded, up to multiplicative constant, by \(T\).
 
With the help of the Cauchy theorem, we can write 
\begin{multline}\label{Integral}
\frac1{2\pi \ic} \int_{(\log x)^{-1}-\ic T}^{(\log x)^{-1}+\ic T} \mathcal{F}_{p}(s) \frac{x^s}{s^2} \dd s
\\
 = \sum_{j=0}^{8} \sum_{\substack{\ell\in {\mathbb Z}\\ \abs{s_{j, \ell}}<T}} \text{Res}\bigg(\mathcal{F}_{p}(s)\frac{x^s}{s^2}, s_{j,\ell}\bigg) 
+ \frac1{2\pi \ic} \int_{\mathscr{L}_{\rm h}\cup\mathscr{L}_{\rm v}} \mathcal{F}_{p}(s) \frac{x^s}{s^2} \dd s,
\end{multline}
where \(\mathscr{L}_{\rm h} \coloneqq[-\eps\pm\ic T, (\log x)^{-1}\pm\ic T]\) and \(\mathscr{L}_{\rm v} \coloneqq[-\eps-\ic T, -\eps+\ic T]\).

Under our assumption on \((a, a^{-1}, b, b^{-1})\), all poles of \(\mathcal{F}_{p}(s)\frac{x^s}{s^2}\) in the rectangle \(-\eps\leq \re s\leq (\log x)^{-1}\) and \(\abs{\tau}\leq T\) are simple except for \(s=0\) which is of order \(3\). The residue of \(\mathcal{F}_p\) at \(0\) is given in Remark~\ref{rem_residue}: %
\[%
\Res\left(\mathcal{F}_p(s),0\right)=\frac{C_{F,p}}{\log p} %
\] 
where %
\[%
C_{F,p}=\frac{2}{(t_a-t_b)^2}\bigg\{\bigg(\frac{1}{4-t_a^2}+\frac{1}{4-t_b^2}\bigg)\bigg(1-\frac{1}{p}\bigg)^2 + \frac{2}{p}\bigg\}. %
\]
It follows that
\begin{equation}\label{residu}
\begin{cases}
\Res\big(\mathcal{F}_{p}(s)\frac{x^s}{s^2}, 0\big) = \frac{C_{F,p}}{\log p} (\log x)^2 + O_{F,p}(\log x),
\\\noalign{\vskip 1mm}
\Res\big(\mathcal{F}_{p}(s)\frac{x^s}{s^2}, s_{j,\ell}\big) \ll_{F, p} 1
\quad
\text{(\(j=0\), \(\ell\neq0\) or \(1\leq j\leq 8\), \(\ell\in\Z\)).}
\end{cases}
\end{equation}

Next we handle the integral over the vertical segment \(\mathscr{L}_{\rm v}\). We have \(\abs{1-\eta_j p^{-s}}\geq p^{\eps}-1\) for \(s\in \mathscr{L}_{\rm v}\) and \(0\leq j\leq 8\). This implies that \(\mathcal{F}_p(s)\ll_{F,p,\eps} 1\) for \(s\in \mathscr{L}_{\rm v}\). Thus 
\begin{equation}\label{vertical}
\frac1{2\pi \ic} \int_{\mathscr{L}_{\rm v}} \mathcal{F}_{p}(s) \frac{x^s}{s^2}\dd s\ll_{F, p,\eps} 1.
\end{equation}

Finally we estimate the contribution of the integral over the horizontal segments \(\mathscr{L}_{\rm h}\). For \(s\in \mathscr{L}_{\rm h}\) and \(0\leq j\leq 8\), we have
\[
\abs{1-\eta_jp^{-s}}
= \abs*{1-p^{-\sigma}\ee^{\ic(\theta_j\mp\theta_{*})}}
\geq p^{-(\log x)^{-1}}\abs{\sin(\theta_j\mp\theta_{*})}
\gg_{F, p} 1.
\]
This implies that \(\mathcal{F}_p(s)\ll_{F,p} 1\) for \(s\in \mathscr{L}_{\rm h}\). Thus 
\begin{equation}\label{horizontal}
\frac1{2\pi \ic} \int_{\mathscr{L}_{\rm h}} \mathcal{F}_{p}(s) \frac{x^s}{s^2}\dd s\ll_{F, p} \frac{1}{T^2}.
\end{equation}

By combining \eqref{infty}, \eqref{Integral}, \eqref{residu}, \eqref{vertical}, \eqref{horizontal} with \eqref{Perron}, we obtain
\[
\sum_{p^\nu\leq x} \lambda_F(p^\nu)^2 \log\bigg(\frac{x}{p^{\nu}}\bigg)= \frac{C_{F,p}}{\log p} (\log x)^2 + O_{F,p}(\log x + T),
\]
which implies the desired asymptotic formula \eqref{SecondMoment} since \(T\asymp \log x\).

The estimate \eqref{FirstMoment} can be proved similarly. %

\section{Proof of Theorem \ref{thm3}}

The notion of \(\mathscr{B}\)-free numbers, as a generalization of square free numbers, was introduced by Erd\H os \cite{Er66}. For a set of integers
\[
\mathscr{B}=\{b_i\colon 1<b_1<b_2<\dots\,\}
\]
such that
\begin{equation}\label{Bhypothese}
\sum_{i\geq 1} 1/b_i<\infty
\qquad\text{and}\qquad
(b_i,b_j)=1
\quad
(i\neq j),
\end{equation}
one says that \(n\geq 1\) is \(\mathscr{B}\)-free if it is not divisible by any element in \(\mathscr{B}\). Many authors studied the distribution of \(\mathscr{B}\)-free integers. A detailed historical description can be found in \cite{KowalskiRobertWu2007, WuZhai2013}. In particular the authors of these two papers proved the following results (see \cite[Corollary 10]{KowalskiRobertWu2007} and \cite[Proposition 2]{WuZhai2013}, respectively):
\begin{itemize}
\item For all \(\eps>0\), \(x\geq x_0(\eps)\) and \(y\geq x^{7/17+\eps}\), we have
\begin{equation}\label{krw}
\#\{x<n\leq x+y\colon \text{\(n\) is \(\mathscr{B}\)-free}\}\gg_{\mathscr{B},\eps} y.
\end{equation}
\item For all \(\eps>0\), \(x\geq x_0(\eps)\), \(y\geq x^{17/38+100\eps}\), \(1\leq a\leq q\leq x^\eps\) with \(\left((a, q), b\right)=1\) for all \(b\in \mathscr{B}\), we have
\begin{equation}\label{wz}
\#\{x<n\leq x+y\colon\text{\(n\equiv a\pmod{q}\) and \(n\) is \(\mathscr{B}\)-free}\}\gg_{\mathscr{B},\eps} y/q.
\end{equation}
\end{itemize}
Now take
\[
\mathscr{B}_{F}\coloneqq\mathfrak{P}_F\cup \{p^2\colon p\in \prem\sset\mathfrak{P}_F\}
\]
where 
\[
\mathfrak{P}_F\coloneqq\{p\colon\lambda_F(p)=0\}.
\]
With the help of \eqref{KowalskiSaha}, it is easy to check that \(\mathscr{B}_{F}\) satisfies the condition \eqref{Bhypothese}. Clearly if \(n\) is \(\mathscr{B}_{F}\)-free, then certainly \(n\) is square free. Since \(\lambda_F(n)=a_F(n)\) for all square free integers \(n\) and they are multiplicative when restricted on these integers, our choice of \(\mathscr{B}_{F}\) guarantees that \(\lambda_F(n)=a_F(n)\neq0\) if \(n\) is \(\mathscr{B}_{F}\)-free. Thus \eqref{krw} and \eqref{wz} imply the first and second assertions of Theorem \ref{thm3}, respectively.

\section{Appendix: Non-vanishing of Hecke eigenvalues for Siegel cusp forms}

\begin{center}
  \textbf{ Appendix by E. Kowalski, ETH Z\"urich, and A. Saha,
    University of Bristol}
\end{center}

We prove the following result:

\begin{theoremAN}
Let \(F\) be a Siegel cusp form of genus \(2\) and level \(1\) which is a Hecke eigenform. Let \(\lambda_F(p)\) denote the normalized \(p\)-th Hecke eigenvalue of \(F\). Then we have
\[
\#\{p\leq x\colon\lambda_F(p)=0\}=o(\pi(x))
\]
as \(x\rightarrow +\infty\), and in fact there exists \(\delta>0\) such that
\[
\#\{p\leq x\colon \lambda_F(p)=0\}\ll \frac{x}{(\log x)^{1+\delta}}.
\]
\end{theoremAN}

Although the argument is short, it depends on a combination of extremely deep results. More precisely, we argue along the same lines as Serre's proof in the case of classical holomorphic modular forms; we need then to use both results of Weissauer~\cite{weissauer} constructing the Galois representations attached to \(F\), as well as those of Dieulefait~\cite{dieulefait} studying the images of these representations.  Crucially, we use the fact that some of the conditional statements of Dieulefait are now known unconditionally due to the proof of Serre's conjectures on modular two-dimensional Galois representations over finite fields~\cite{khare-wintenberger} and the proof of the functorial transfer of full level Siegel eigenforms to \(\GL_4\)~\cite{PSS2012}.

\begin{proof}
Let \(F\) be as in the theorem. If \(F\) is a Saito-Kurokawa lift, then \(\lambda_F(p)>0\) for all primes \(p\). So we may assume that  \(F\) is not a Saito-Kurokawa lift. We denote by \(a_F(p)= p^{k-\frac32} \lambda_F(p)\) the unnormalized Hecke eigenvalue of \(F\) at \(p\), so that
\[
\#\{p\leq x\colon\lambda_F(p)=0\}=\#\{p\leq x\colon a_F(p)=0\},
\]
and we will bound the latter.
\par
\textbf{Step 1.} It is well-known (see~\cite{NPS, As-S}) that the representation \(\Pi_F\) of \(\GSp_4(\A)\) generated by the adelization of \(F\) is irreducible\footnote{This is not the same as asserting multiplicity one for \(\GSp_4\), which to the best of our knowledge remains open.}, unitary, cuspidal and belongs to the discrete series of weight \((k,k)\) at infinity. By work of Weissauer, there exists a number field\footnote{Here, the number field is chosen large enough to ensure that for all primes \(\lambda |l\) of \(E\), the representation \(\rho_{F,\lambda}\) is defined over \(E_\lambda\).} \(E/\Q\), such that for any prime number \(l\) and any extension \(\lambda\) of \(l\) to \(E\), there is a semisimple Galois representation
\[
 \rho_{F,\lambda}\colon\Gal(\bar{\Q}/\Q)\rightarrow \GL_4(E_{\lambda})
\]
having the property that at all primes \(p\neq\ell\), \(\rho_{F,\lambda}\) is unramified at \(p\) and furthermore we have
\[
\Tr(\rho_{F,\lambda}(\frob_p))=a_p.
\]
\par
\textbf{Step 2.} A key corollary~\cite[Thm. 5.1.4]{PSS2012} of the functorial transfer of full level Siegel eigenforms to \(\GL_4\) is that \(\Pi_F\) is weakly equivalent to a generic representation. It follows from~\cite[Thm. IV]{weissauer} that the image of \(\rho_{F,\lambda}\) is contained in \(\GSp_4(E_{\lambda})\). Since the image is compact, it follows that  up to conjugation, the image of \(\rho_{F,\lambda}\) is in fact contained in \(\GSp_4(\mathcal{O}_{E_\lambda})\).
\par
\textbf{Step 3.} We fix a prime \(\ell>2k-2\) \emph{totally split} in \(E\). Henceforth, we fix any place \(\lambda\) above \(\ell\) and refer to \(\rho_{F,\lambda}\) as \(\rho_{F,\ell}\); its image is contained inside \(\GSp_4(\Z_\ell)\). Let \(\bar{\rho}_{F,\ell}\) be the projective representation
\[
\bar{\rho}_{F,\ell}\colon\Gal(\bar{\Q}/\Q)\rightarrow \PGSp_4(\Z_{\ell})
\]
obtained from \(\rho_{F,\ell}\), and let \(G\) be its image. This is an \(\ell\)-adic analytic Lie group. Let
\[
C=G\cap \{ g\in\PGSp_4(\Z_{\ell})\colon \Tr(g)=0\},
\]
(where the trace is computed for any lift in \(\GSp_4(\Z_{\ell})\), which is well-defined), a conjugacy-invariant analytic subvariety of \(G\) of codimension \(1\).
\par
\textbf{Step 4.} By a result of Dieulefait~\cite[Th. 4.2]{dieulefait} (see below for details), we have \(G=\PGSp_4(\Z_{\ell})\) for almost all \(l\), so we assume that our \(l\) has this property. Thus \(\dim G=10\) and \(\dim C=9\). We have therefore
\[
\#\{p\leq x\colon a_F(p)=0\}= \#\{p\leq x, p\neq l\colon \bar{\rho}_{F,\ell}(\frob_p)\in C\}+\delta_{a_F(l),0}
\]
and by~\cite[Th. 10]{Serre1981}, we obtain 
\[
\#\{p\leq x\colon a_F(p)=0\}\ll\frac{x}{(\log x)^{1+\delta}}
\]
for any \(\delta<1/10\).
\end{proof}


\begin{remarkAN}
  Dieulefait proceeds by considering the residual mod \(\ell\) reduction of  \(\bar{\rho}_{F, \ell}\), whose image (in our case) lies inside the finite field \(\PGSp_4(\F_\ell)\). If the image is not the full group then it can be classified into ten cases (see~\cite[Sec. 3.1]{dieulefait}) of which cases 9) and 10) cannot occur for us since the finite field is prime. The remaining cases are eliminiated by invoking either the Ramanujan bound or the modularity of two-dimensional Galois representations over finite fields (Serre's conjecture); the latter has been proved by Khare and Wintenberger~\cite{khare-wintenberger}. Note that we do not need the ``untwisted" hypothesis of Dieulefait~\cite[Def. 4.1]{dieulefait} since this hypothesis was only used by him to eliminate cases 9) and 10).
\end{remarkAN}

\begin{remarkAN}
  In our main theorem, the words ``Hecke eigenform'' can be relaxed to ``Hecke eigenform at almost all primes''. Indeed, if \(F\) is a Siegel cusp form of genus \(2\) and level \(1\) which is a Hecke eigenform at almost all primes, then it is automatically an eigenform for all Hecke operators, see~\cite[Cor. 3.4]{NPS}.
\end{remarkAN}

\begin{remarkAN}
  The result does not extend to Siegel cusp forms of arbitrary level \(N\geq 1\) (just as the case of classical cusp forms must exclude the CM forms, which have level larger than \(1\)). More precisely, given a quadratic field \(K/\Q\), and a suitable Hilbert or Bianchi cusp form (see~\cite{JLR} and~\cite{Ber}) \(f\) over \(K\), one can construct a Siegel cusp form \(F\) over \(\Q\) with spinor \(L\)-function given by
\[
L(s,F)=L(s,f)
\]
where the \(L\)-functions are all Langlands-normalized. Note here that \(L(s,f)\) has degree 2 over \(K\) and hence degree 4 over \(\Q\). In particular, for any prime \(p\) such that \(p\) is inert in \(K\) the local \(L\)-factor at \(p\) is a polynomial in \(p^{-2s}\), thus showing that \(\lambda_F(p)=0\) for at least half the primes.
\par
In fact, one can show that this type of example, as well as certain types of Yoshida lifts, are the only examples of Siegel cusp forms of genus \(2\) where the spinor \(L\)-function has a positive density of zero coefficients at primes. We will come back to this in a later paper.
\end{remarkAN}

\providecommand{\bysame}{\leavevmode\hbox to3em{\hrulefill}\thinspace}
\providecommand{\MR}{\relax\ifhmode\unskip\space\fi MR }
\providecommand{\MRhref}[2]{%
  \href{http://www.ams.org/mathscinet-getitem?mr=#1}{#2}
}
\providecommand{\href}[2]{#2}


\end{document}